\documentclass[11pt]{article}
\usepackage{amssymb,latexsym,amsmath,amsbsy,amsthm,amsxtra,amsgen,dsfont,pdfsync}
\usepackage{color}
\usepackage{hyperref}
\usepackage{graphicx}

\newfont{\msbm}{msbm10 at 11pt}
\newcommand {\R} {\mbox{\msbm R}}

\newcommand {\E} {\mbox{\msbm E}}
\newcommand {\1} {\mathds{1}}

\def\P{\ensuremath{\mathbf{P}}}

\newfont{\msbmsm}{msbm10 at 8pt}

\def\eps{\varepsilon}

\newtheorem{Theo}{Theorem}
\newtheorem{Lemma}[Theo]{Lemma}

\newtheorem{Prop}[Theo]{Proposition}

\newtheorem{Rmk}[Theo]{Remark}

\begin{document}
\title{The all-time maximum for branching Brownian motion with absorption conditioned on long-time survival}

\author{Pascal Maillard\thanks{Universit\'e de Toulouse and Institut Universitaire de France. Supported in part by ANR grant ANR-20-CE92-0010-01.} \, and Jason Schweinsberg\thanks{University of California San Diego.  Supported in part by NSF Grant DMS-1707953}}
\maketitle

\begin{abstract}
We consider branching Brownian motion in which initially there is one particle at $x$, particles produce a random number of offspring with mean $m+1$ at the time of branching events, and each particle branches at rate $\beta = 1/2m$.  Particles independently move according to Brownian motion with drift $-1$ and are killed at the origin.  It is well-known that this process eventually dies out with positive probability.  We condition this process to survive for an unusually large time $t$ and study the behavior of the process at small times $s \ll t$ using a spine decomposition.  We show, in particular, that the time when a particle gets furthest from the origin is of the order $t^{5/6}$.

Keywords: branching Brownian motion, Yaglom limit theorem, Brownian taboo process, excursion theory

MSC 2020 classification: 60J80, 60J70, 60F05 
\end{abstract}

\section{Introduction}

We will consider one-dimensional branching Brownian motion with absorption.  This process was first studied in 1978 by Kesten \cite{kesten}.  In addition to its intrinsic mathematical interest, branching Brownian motion with absorption has been applied in the study of partial differential equations \cite{hhk06} and has been used to model populations undergoing selection in \cite{bbs, bdmm1, bdmm2}.

The process evolves as follows.  At time zero, there is a single particle at $x > 0$.  Each particle moves according to one-dimensional Brownian motion with a drift of $-\mu$.  Each particle independently branches at rate $\beta$, and when a branching event occurs, the particle dies and is replaced by a random number of offspring.  We assume the numbers of offspring produced at different branching events are independent and identically distributed, and we denote by $p_k$ the probability that an individual has $k$ offspring.  We define $m$ so that $m + 1 = \sum_{k=0}^{\infty} k p_k$ is the mean of the offspring distribution, and we assume the offspring distribution has finite variance.  We also assume that $\beta = 1/2m$, which by a scaling argument can be done without loss of generality.

Kesten \cite{kesten} showed that if $\mu \geq 1$, then the process almost surely goes extinct, whereas if $\mu < 1$, then with positive probability, the process survives forever.  We will assume that $\mu = 1$, which is the case of critical drift.  Let $\zeta$ denote the time when the process goes extinct, which is almost surely finite.  It was shown in \cite{ms20}, building on earlier work in \cite{kesten, bbs14}, that there is a positive constant $C$ such that 
\begin{equation}
\label{eq:survival}
\P_x(\zeta > t) \sim C x e^{x - (3 \pi^2 t/2)^{1/3}},
\end{equation}
 where $\sim$ means that the ratio of the two sides tends to one as $t \rightarrow \infty$ while $x$ is fixed.

Yaglom-type limit theorems for the behavior of the process conditioned on the event that it survives for an unusually long time were first proved by Kesten \cite{kesten}.  Kesten obtained estimates on the number of particles at time $t$ and the position of the right-most particle at time $t$ for the process conditioned on survival until time $t$.  More precise results along these lines, as well as other results about the behavior of the process conditioned on survival until time $t$,
were recently established in \cite{ms20}.  The main results in \cite{ms20} focused on the behavior of the process either at times close to $t$, or at times in $[\delta t, (1 - \delta)t]$ for a fixed constant $\delta > 0$.  In the present paper, we will be concerned primarily with the behavior of the branching Brownian motion at times $s \ll t$, when the process is conditioned to survive for an unusually long time $t$.

We now introduce some notation. We use the standard Ulam--Harris labelling, in that particles are labeled by words over the positive integers. Let $N(s)$ denote the set of particles alive at time $s$.  If $u \in N(s)$, then $X_u(s)$ denotes the position of the particle $u$ at time $s$, and for $0 < r < s$, we denote by $X_u(r)$ the position at time $r$ of the ancestor of this particle.  Let $X(s) := \{X_u(s): u \in N(s)\}$ be the set of the locations of all particles at time $s$, and let $$M(s) := \max_{u \in N(s)} X_u(s)$$ be the location of the particle furthest from the origin at time $s$.
For $s \in [0, t]$, let $$L_t(s) := c (t - s)^{1/3}, \qquad c := \bigg( \frac{3 \pi^2}{2} \bigg)^{1/3}.$$  Roughly speaking, the significance of $L_t(s)$ is that if a particle reaches $L_t(s)$ at time $s$, then there is a good chance that a descendant of this particle will still be alive at time $t$.  See \cite{bbs14, ms20} for more precise versions of this statement.  To simplify notation, we also write
$$L_t := L_t(0) = ct^{1/3}.$$  Because, throughout the paper, we will be considering asymptotics as $t \rightarrow \infty$, we will use the notation $f(t) \sim g(t)$ to mean $\lim_{t \rightarrow \infty} f(t)/g(t) = 1$ and $f(t) \ll g(t)$ to mean $\lim_{t \rightarrow \infty} f(t)/g(t) = 0$.  We will write $f(t) \asymp g(t)$ if there exist constants $0 < C_1 < C_2 < \infty$ such that $C_1f(t) < g(t) < C_2f(t)$ for all $t > 0$.
We will also say that a random variable $Y_t$, which depends on $t$, is $\Theta_p(f(t))$ if for all $\eps > 0$, there exist constants $0 < C_1 < C_2 < \infty$ such that $P(C_1 f(t) < Y_t < C_2 f(t)) > 1 - \eps$ for sufficiently large $t$. Convergence in law is denoted by $\Rightarrow$, and convergence in probability is denoted by $\rightarrow_p$.

\subsection{Main Result}

Our main result, Theorem \ref{th:global_max} below, concerns the maximum location that any particle achieves before time $t$, as well as the time when this all-time maximum occurs, when the process is conditioned to survive until time $t$.  Note, in particular, that the time at which the all-time maximum is achieved is $\Theta_p(t^{5/6})$, which is much smaller than $t$.  This occurs because, conditional on survival of the process for a large time $t$, with high probability a particle quickly moves very far away from the origin, getting close to $L_t$.  After this initial burst, the value of $M(s)$ decreases over time as the process heads towards its extinction shortly after time $t$.  See Figure 5 of \cite{ds07} for an illustration of this phenomenon when the drift $\mu$ is slightly larger than one.  See also Theorems~1.5 and 2.9 of \cite{ms20} for results about the asymptotic behavior of $M(s)$ when $s \geq \delta t$, conditional on survival of the process until time $t$.

\begin{Theo}
\label{th:global_max}
Define
\[
\mathfrak M = \max_{s\ge0} M(s),\quad \mathfrak m = \operatorname*{arg\,max}_{s\ge0} M(s).
\]
Suppose the position $x$ of the initial particle may depend on $t$ but satisfies $L_t - x \gg t^{1/6}$.  Then as $t \rightarrow \infty$, conditional on $\zeta > t$ we have the convergence in law
\begin{equation}\label{globmax}
\left(\frac{L_t - \mathfrak M}{t^{1/6}}, \frac{\mathfrak m}{t^{5/6}}\right) \Rightarrow \left(c^{1/2}R, 3c^{-1/2}UR\right),
\end{equation}
where $R$ and $U$ are independent random variables, $R$ is Rayleigh distributed with density $2r e^{-r^2}$ on $\R_+$, and $U$ has a uniform distribution on $[0,1]$.
\end{Theo}

\begin{Rmk}
Note that the convergence in (\ref{globmax}) can also be written as
\begin{equation}\label{maxrem}
\left(L_t^{-1/2}(L_t - \mathfrak M), L_t^{1/2} \frac{\mathfrak m}{t}\right) \Rightarrow (R, 3UR).
\end{equation}
\end{Rmk}

The proof of Theorem \ref{th:global_max} can be found in section \ref{maxsec}.

\subsection{Related work and comments}

The branching Brownian motion with absorption (and drift) considered in this article is a particular example of a \emph{branching Markov process}, as defined for example in \cite{inw}. Since we are interested in the case of critical drift, it is interesting to compare our results with classical results on \emph{critical} branching processes. The behavior of critical branching processes is well-understood in the case of the classical mono-type branching process, i.e.~the Bienaymé--Galton--Watson process.  Under the assumption that the offspring distribution has finite variance, classical results by Kolmogorov \cite{Kolmogorov1938} and Yaglom \cite{yaglom} show that the probability of survival until generation $n$ decays inversely proportional to $n$ and conditioned on survival, the size of the population, rescaled by $1/n$, converges in law to an exponential distribution. Slack~\cite{Slack1968} extended these results to offspring distributions in the domain of attraction of an $\alpha$-stable law, with $\alpha\in (1,2]$, where the survival probability decays like $1/n^{\alpha-1}$. The books by Harris~\cite{HarrisBook} and Athreya and Ney~\cite{AthreyaNeyBook} contain extensions to finite-type branching processes, under a certain finite-variance condition. Harris~\cite{HarrisBook} also treats a certain class of processes with a countably infinite number of types, which was recently generalized by de Raphélis~\cite{deRaphelis2017,deRaphelis2022}. The latter articles actually prove convergence of the rescaled tree to a limiting tree.  This was done previously for the finite-type case by Miermont~\cite{Miermont2008} ($\alpha=2$) and Berzunza~\cite{Berzunza2018} ($\alpha\in (1,2]$).

Critical branching Markov processes in general type space (and in continuous time) were studied in the 70's by Hering and co-authors \cite{Hering1971,Hering1977,Hering1978,HH1981}. Under a condition on the first moment semigroup called \emph{uniform primitivity}, they obtain asymptotics for the survival probability of order $1/t^{\alpha-1}$, for $\alpha \in (1,2]$, as well as a Yaglom limit theorem. See also Harris, Horton, Kyprianou and 	Wang \cite{HHKW2022} for a different approach using moments and many-to-few formulae, under more restrictive assumptions. The condition of uniform primitivity is verified for branching diffusions on bounded domains under regularity assumptions on the coefficients and the domain, but it is in general not verified for branching diffusions on unbounded domains, such as the one considered in this article.  The notion of \emph{generalized principal eigenvalue} $\lambda_c$ has been succesfully applied to the question of \emph{local} survival vs. local extinction for general branching diffusions on unbounded domains in \cite{EK2004}, and to the generalization of limit theorems in the supercritical case $\lambda_c > 0$ in, for example, \cite{EHK2010}. However, refined limit theorems are missing to this date in the critical case $\lambda_c = 0$. 

In light of these results, the process studied in this article can be considered as a prototypical example of a branching diffusion in an unbounded domain and whose asymptotics are radically different from the cases mentioned above. Indeed, it corresponds to the case $\alpha = 1$, as witnessed by the streched exponential asymptotic for the survival probability from \eqref{eq:survival} and the appearance of an underlying continuous-state branching process driven by a $1$-stable Lévy process, called Neveu's continuous state branching process \cite{ms20}. The results in this article give a precise picture of the behavior of the process conditioned on survival until time $t$.  Underlying Theorem~\ref{th:global_max} is the maybe surprising fact that the process conditioned on survival until time $t$ behaves radically differently from the process conditioned to survive forever\footnote{By this we mean the process conditioned to survive until time $s$ and letting $s\to\infty$.}. Indeed using \eqref{eq:survival} and classical change-of-measure techniques for branching processes (see below), one can show that the latter can be constructed as a BBM with a distinguished particle commonly called a \emph{spine}, moving according to a Bessel-3 process. As seen below, the process conditioned to survive until time $t$ also behaves like a certain spinal process on time-scales small than $t$, but the two spinal processes are comparable only until the time-scale $t^{2/3}$, when their behavior drastically changes. As shown below, this leads to the global maximum appearing at a time-scale $t^{5/6} \ll t$, as shown in Theorem~\ref{th:global_max}.

\subsection{An auxiliary branching Brownian motion with spine}

To prove Theorem \ref{th:global_max}, we will introduce another process, a branching Brownian motion (BBM) with a distinguished particle called the \emph{spine}, which will approximate the original process conditioned on survival until time $t$. Its definition involves the so-called \emph{Brownian taboo process}, introduced by Knight \cite{knight}. 

\paragraph{Brownian taboo process.} Following Knight \cite{knight}, consider standard Brownian motion started at some point $x\in (0,1)$ and killed as soon as it exits the interval $[0,1]$. Then the function $x\mapsto \sin(\pi x)$ is a non-negative eigenfunction of the infinitesimal generator of this process with eigenvalue $-\pi^2/2$. The \emph{Brownian taboo process} is the stochastic process $(K_t)_{t\ge0}$ defined as the Doob $h$-transform of this process with respect to this eigenfunction. Its law satisfies that for every $t\ge0$ and every bounded or positive measurable functional $F$,
\begin{equation}
\label{eq:taboo} \E_x[F((K_s)_{s\in [0,t]})] = \frac{1}{\sin(\pi x)}e^{\frac{\pi^2}{2}t}\E_x[F((B_s)_{s\in[0,t]})\sin(\pi B_t)],
\end{equation}
where $(B_t)_{t\ge0}$ is standard Brownian motion started from $x$. The Brownian taboo process can be interpreted as Brownian motion conditioned to stay forever in the interval $(0,1)$.

It follows from its definition as a Doob transform of Brownian motion that the Brownian taboo process is a diffusion on $(0,1)$ with speed measure $m(dx)$ and scale measure $s(dx)$ given by (see e.g.~\cite[Paragraph II.31]{borodinsalminen})
\[
m(dx) = 2\sin^2(\pi x)\,dx,\quad s(dx) = \frac{1}{\sin^2(\pi x)}\,dx.
\]
Note that $\int_0^1m(dx) = 1$, so that $m$ is the stationary probability for the process. One can check using the formulae in \cite[Paragraph II.6]{borodinsalminen} that the boundary points 0 and 1 are both \emph{entrance-not-exit}, i.e.~the process can be defined starting from 0 or 1 as the weak limit when $x\to0$ or $x\to 1$, but the process never hits $0$ nor $1$ when started inside the interval $(0,1)$. 

\paragraph{Branching Brownian motion with spine.} We define a branching Brownian motion with a distinguished particle called the \emph{spine} as follows:
\begin{itemize}
	\item The process starts at time zero from a single particle, the spine, at a point $x\in [0,L_t]$.
	\item For $s \in [0, t]$, define $$\tau(s) = \int_0^s \frac{1}{L_t(u)^2} \: du.$$  The spine's trajectory is equal in law to the process $(L_t(s) K_{\tau(s)})_{s\in[0,t]}$, where $(K_u)_{u\ge0}$ is the Brownian taboo process on $[0,1]$ started at $x/L_t(0)$.
	\item The spine branches at the accelerated rate $(m+1)\beta$ and according to the size-biased offspring distribution, i.e.~the probability of it having $k$ offspring is $kp_k/(m+1)$, $k=1,2,\ldots$. These offspring are located at same position as their parent.
	\item One of these offspring is chosen randomly to continue as the spine. The others spawn independent branching Brownian motions with drift $-1$ and absorption at $0$ (i.e., independent copies of the process $X$ started from their position).
\end{itemize}
See e.g.~Hardy and Harris \cite{HH} or Harris and Roberts \cite{HR} for a general theory of spine decompositions for branching Markov processes, including a rigorous construction.

For the BBM with spine, we let ${\tilde N}(s)$ be the set of particles alive at time $s$.  For $u \in {\tilde N}(s)$, we let ${\tilde X}_u(s)$ be the location of the particle $u$ at time $s$, and we let ${\tilde X}(s) := \{{\tilde X}_u(s): u \in {\tilde N}(s)\}$.  Let $\xi_s$ be the spine particle at time $s$, and denote the trajectory of the spine by $({\tilde X}_\xi(s))_{s\in[0,t]}$.

The following proposition, which is proved in section \ref{spinesec}, says that up to time $\delta t$ for small $\delta$, the BBM with spine approximates well the BBM conditioned on survival until time $t$.

\begin{Prop}
\label{lem:spine_comparison}
Let $\mu_s$ be the law of $(X(r))_{r\in[0,s]}$ conditioned on $\{\zeta > t\}$, and let $\nu_s$ be the law of $(\tilde X(r))_{r\in[0,s]}$, where both $(X(r))_{r\in[0,s]}$ and $(\tilde X(r))_{r\in[0,s]}$ begin with one particle at $x$.  We assume that $x$ may depend on $t$ but satisfies $\lim_{t \rightarrow \infty} (L_t - x) = \infty$.  Then for every $\eps>0$, there exists $\delta > 0$ such that for sufficiently large $t$, we have
\[
d_{TV}(\mu_{\delta t},\nu_{\delta t}) \le \eps,
\]
where $d_{TV}$ is the total variation distance between measures.
\end{Prop}

\subsection{Sketch of the argument for Theorem \ref{th:global_max} based on excursion theory}

We outline here an argument for Theorem \ref{th:global_max} based on the spine construction and excursion theory for the Brownian taboo process.  Let $(K_s)_{s\ge0}$ be a Brownian taboo process starting from an arbitrary point in the interval $(0,1)$.  Define for $\gamma > 0$,
$$M_\gamma = \min_{s\ge0} \, \big(K_s + \gamma s\big),\qquad m_\gamma = \operatorname*{argmin}_{s\ge0} \, \big(K_s + \gamma s \big).$$
We claim that as $\gamma\to 0$, we have the convergence in law
\begin{equation}\label{Rayleighsimple}
\bigg(\frac{M_\gamma}{\sqrt \gamma}, \sqrt \gamma m_\gamma \bigg) \Rightarrow ({\tilde R},U {\tilde R}),
\end{equation}
where ${\tilde R}$ and $U$ are independent random variables, ${\tilde R}$ has a Rayleigh distribution with density $\pi^2 r e^{-\pi^2r^2/2}$ on $\R_+$, and $U$ has a uniform distribution on $[0,1]$.  A stronger version of (\ref{Rayleighsimple}) will be proved in Lemma \ref{lem:taboo_with_drift} below.

We can understand why (\ref{Rayleighsimple}) should hold by using the excursion theory for the Brownian taboo process, which was developed by Lambert \cite{lambert} and will be reviewed in section \ref{taboosec}.  The Brownian taboo process stays in the interval $(0,1)$, and it can be decomposed according to its excursions away from $1/2$.  We will consider the Poisson point process $\mathcal{N}$ consisting of the points $(u, a_u)$, where $u$ is the local time at $1/2$ associated with a particular excursion of the taboo process below $1/2$, and $a_u$ is the minimum value achieved by the taboo process during this excursion.  As we will see in (\ref{hdlimit}) below, for small $d$, the rate of excursions per unit of local time during which the process goes below $d$ is approximately $\pi^2 d/2$.  Therefore, for small values of the second coordinate, the intensity of the Poisson point process $\mathcal{N}$ is approximately $\pi^2/2$.

When $\gamma$ is small, we expect $M_{\gamma}$ to be close to zero, so we will be concerned primarily with excursions of the Brownian taboo process that get close to zero.  The event $M_{\gamma}/\sqrt{\gamma} \leq x$ is the same as the event that $K_s \leq x \sqrt{\gamma} - \gamma s$ for some $s \geq 0$.  
Also, we will see in (\ref{Ls12}) below that after a large time $s$, the Brownian taboo process will have accumulated a local time of approximately $2s$ by time $s$.  Therefore, if $K_s \leq x \sqrt{\gamma} - \gamma s$, then the excursion underway at time $s$ should correspond to a point $(u, a_u)$ of $\mathcal{N}$ such that $a_u \leq x \sqrt{\gamma} - \gamma u/2$.  That is, we have $M_{\gamma}/\sqrt{\gamma} \leq x$ when there is a point of $\mathcal{N}$ in the triangle in the figure below.  Because this triangle has area $x^2$ and the Poisson point process $\mathcal{N}$ has intensity approximately $\pi^2/2$, we have $P(M_{\gamma}/\sqrt{\gamma} > x) \approx e^{-\pi^2 x^2/2}$, which is the Rayleigh distribution that appears in (\ref{Rayleighsimple}).  Furthermore, conditional on $M_{\gamma}/\sqrt{\gamma} = x$, the excursion that produces the minimum value of $K_s - \gamma s$ should correspond to a point located at a uniform position on the diagonal line in the figure below.  The local time $u$ of this excursion is therefore approximately uniformly distributed on $[0, 2x/\sqrt{\gamma}]$, which means the actual time $s$ of the excursion is approximately uniformly distributed on $[0, x/\sqrt{\gamma}]$, as indicated in (\ref{Rayleighsimple}).

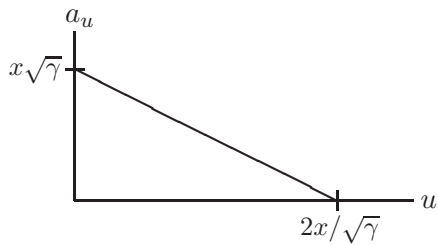
\begin{center}
\setlength{\unitlength}{.25cm}
\thicklines
\begin{picture}(24,12.5)
\put(4,2){\line(1,0){18}}
\put(4,2){\line(0,1){9}}
\put(18,2){\line(-2,1){14}}
\put(22.4,1.6){$u$}
\put(0.5,8.7){\small{$x\sqrt{\gamma}$}}
\put(18,1.5){\line(0,1){1}}
\put(3.5,9){\line(1,0){1}}
\put(3.5,11.5){$a_u$}
\put(16,0.2){{\small $2x/\sqrt{\gamma}$}}
\end{picture}
\end{center}

We now sketch an argument for how to obtain Theorem \ref{th:global_max} from (\ref{Rayleighsimple}).  More details are provided in section \ref{maxsec}. We will argue that the location $M(s)$ of the particle at time $s$ that is furthest from the origin is likely to stay close to the location $X_{\xi}(s)$ of the spinal particle.  We can write $X_{\xi}(s) = L_t(s) (1-K_{\tau(s)})$, where $(K_u)_{u \geq 0}$ is a Brownian taboo process (note that $K$ is a Brownian taboo process if and only if $1-K$ is a Brownian taboo process).  Then
\begin{equation}\label{LX}
\frac{L_t - X_{\xi}(s)}{L_t} = \bigg(1 - \frac{L_t(s)}{L_t} \bigg) + \frac{L_t(s)}{L_t} K_{\tau(s)}.
\end{equation}

The function $\tau$ is strictly increasing on $[0, t]$ with $\tau(t) = (3/c^2) t^{1/3}$, so the inverse function $\tau^{-1}$ is well-defined on the interval $[0, (3/c^2) t^{1/3}]$.  Allowing $s$ and $u$ to vary with $t$ for this paragraph, for $s \ll t$ we have
\begin{equation}\label{tauasymp}
\tau(s) \sim s/L_t^2,
\end{equation}
and therefore if $u \ll t^{1/3}$, then
\begin{equation}\label{tauinverse}
\tau^{-1}(u) \sim L_t^2 u.
\end{equation}
Also, for $s \ll t$, we have $$L_t - L_t(s) \sim \frac{cs}{3t^{2/3}} = \frac{c^3}{3} \cdot \frac{s}{L_t^2} = \frac{\pi^2}{2} \cdot \frac{s}{L_t^2} \sim \frac{\pi^2 \tau(s)}{2}.$$
It follows that for $u \ll t^{1/3}$, we have
\begin{equation}\label{Ltdiff}
L_t - L_t(\tau^{-1}(u)) \sim \frac{\pi^2 u}{2}.
\end{equation}

We can now make the substitution $u = \tau(s)$, and equation (\ref{LX}) becomes
\begin{equation}\label{LXu}
\frac{L_t - X_{\xi}(\tau^{-1}(u))}{L_t} = \bigg(\frac{L_t - L_t(\tau^{-1}(u))}{L_t} \bigg) + \frac{L_t(\tau^{-1}(u))}{L_t} K_{u}.
\end{equation}
Let $$\gamma = \frac{\pi^2}{2L_t}.$$
Using (\ref{Ltdiff}) to approximate the first term on the right-hand side of (\ref{LXu}) and then approximating $L_t(\tau^{-1}(u))/L_t$ by $1$ in the second term, we obtain the approximation $$\frac{L_t - X_{\xi}(\tau^{-1}(u))}{L_t} \approx \gamma u + K_u.$$
Consequently, as long as the branching Brownian motion attains its all-time maximum at approximately the same time that the spine does, we will be able to obtain from (\ref{Rayleighsimple}) that
\begin{equation}\label{LXconv}
\bigg( \frac{L_t - \mathfrak{M}}{L_t \sqrt{\gamma}}, \sqrt{\gamma} \tau(\mathfrak{m}) \bigg) \Rightarrow ({\tilde R}, U {\tilde R}).
\end{equation}
Note that $L_t \sqrt{\gamma} = \sqrt{\pi^2 L_t/2}$.  Also, $\tau(\mathfrak{m}) \approx \mathfrak{m}/L_t^2$ and
\begin{equation}\label{gammaLt2}
\frac{\sqrt{\gamma}}{L_t^2} = \sqrt{\frac{\pi^2 L_t}{2}} \cdot \frac{1}{L_t^3} = \sqrt{\frac{\pi^2 L_t}{2}} \cdot \frac{1}{c^3 t} = \sqrt{\frac{2}{\pi^2}} \cdot \frac{\sqrt{L_t}}{3t}.
\end{equation} 
Because ${\tilde R}\sqrt{\pi^2/2}$ has the same distribution as $R$, the result (\ref{maxrem}) can be obtained from (\ref{LXconv}), and Theorem \ref{th:global_max} follows.

\section{Proof of Proposition 3}\label{spinesec}

Our goal in this section is to prove Proposition \ref{lem:spine_comparison}.  We begin with the following lemma.

\begin{Lemma}
\label{lem:BM_girsanov}
Let $(B_s)_{s\ge0}$ be a standard Brownian motion started from a point $x\in[0,L_t(0)]$ and let $(b_u)_{u\ge0}$ be a standard Brownian motion started from $x/L_t(0)\in [0,1]$. Define
\[
B'_u = \frac{B_{\tau^{-1}(u)}}{L_t(\tau^{-1}(u))}.
\]
Then, there exists a constant $C>0$ such that for every $\delta\in(0,1/2)$, every $t\ge1$, and every positive bounded measurable functional $F$,
\[
\exp(-C(t^{-1/3}+\delta)) \le \frac{\E\left[F((B'_u)_{u\in[0,\tau(\delta t)]})\1_{B'_u\in[0,1]\,\forall u\le \tau(\delta t)}\right]}{\E\left[F((b_u)_{u\in[0,\tau(\delta t)]})\1_{b_u\in[0,1]\,\forall u\le \tau(\delta t)}\right]} \le \exp(Ct^{-1/3}).
\]
\end{Lemma}

\begin{proof}
Define $\delta$ and $F$ as in the statement of the lemma.
Define for $s\in [0, t)$,
\[
\rho_s = \bigg( \frac{L_t(0)}{L_t(s)} \bigg)^{1/2} \exp \bigg( \frac{L_t'(s) B_s^2}{2L_t(s)} - \frac{L_t'(0)B_0^2}{2L_t(0)} - \int_0^s \frac{L_t''(u)B_u^2}{2L_t(u)} \: du \bigg).
\]
By Lemma~5.3 in \cite{ms20}, we have
\begin{equation}
\E\left[F((b_u)_{u\in[0,\tau(\delta t)]})\1_{b_u\in[0,1]\,\forall u\le \tau(\delta t)}\right] = \E\left[F((B'_u)_{u\in[0,\tau(\delta t)]})\1_{B'_u\in[0,1]\,\forall u\le \tau(\delta t)}\,\rho_{\delta t}\right].
\end{equation}
It remains to show that  for some $C>0$, we have that $e^{-Ct^{-1/3}} \le \rho_{\delta t} \le e^{C(\delta+t^{-1/3})}$ on the event $\{B'_u\in[0,1]\,\forall u\le \tau(\delta t)\}$. On this event, we have $B_s\in[0,L_t(s)]$ for every $s\le \delta t$, by the definition of $B'_u$. Furthermore, note that
\[
L_t(s) =O(t^{1/3}),\quad |L_t'(s)|=O(t^{-2/3}),\quad|L_t''(s)| =O(t^{-5/3})\quad\text{for all $s\le t/2$.}
\]
Hence, the absolute value of the quantity in the exponential in the definition of $\rho_s$ is bounded by $Ct^{-1/3}$ for some $C>0$, when $s=\delta t$. Furthermore, the first factor in the definition of $\rho_s$ satisfies, when $s=\delta t$,
\begin{equation}
\label{eq:L}
1\le \bigg( \frac{L_t(0)}{L_t(\delta t)} \bigg)^{1/2} = \bigg( \frac{1}{1-\delta} \bigg)^{1/6} \le e^{C\delta},
\end{equation}
for some $C>0$. The statement follows.
\end{proof}

For the original branching Brownian motion process, define, as in \cite{ms20},
\begin{align*}
z_t(x,s) &= L_t(s) \sin \bigg( \frac{\pi x}{L_t(s)} \bigg) e^{x - L_t(s)} \1_{x \in [0, L_t(s)]} \\
Z_t(s) &= \sum_{u \in N(s)} z_t(X_u(s), s) \\
Z_t'(s) &= \sum_{u \in N(s)} z_t(X_u(s), s) \1_{X_u(r)\in [0,L_t(r)]\,\forall r\le s}.
\end{align*}
We recall from \cite{ms20} that these quantities control the survival probabilities of the process.  For example, conditional on the process up to time $s$, the probability that the process survives until time $t$ is approximately proportional to $Z_t(s)$, assuming that all particles are far below $L_t(s)$ at time $s$.

\begin{Lemma}
\label{lem:spine_Z}
There exists $C>0$ such that for every $\delta\in(0,1/2)$, $t\ge1$, and $x\in(0,L_t(0))$, and every positive bounded measurable functional $F$, we have
\[
\exp(-C(t^{-1/3}+\delta)) \le \frac{\E_x\left[F((X(s))_{s\in[0,\delta t]})Z_t'(\delta t)\right]}{\E_x\left[F((\tilde X(s))_{s\in[0,\delta t]})\right] z_t(0,x)} \le \exp(C(t^{-1/3}+\delta)).
\]
\end{Lemma}

\begin{proof}
Let $\delta$, $t$, $x$, and $F$ be defined as in the statement of the lemma.
Throughout the proof, we will make use of two auxiliary branching processes with spine, denoted by $X'$ and $X''$, with the position of their respective spines denoted by $X_\xi'$ and $X_\xi''$. These processes are defined in essentially the same way as $\tilde X$, but with the following differences:
\begin{itemize}
\item In $X'$, the spine performs standard Brownian motion started at $x$ and killed at 0.
\item In $X''$, the spine's motion satisfies $X_\xi''(s) = L_t(s)B''_{\tau(s)}$ for a Brownian motion $(B''_u)_{u\ge0}$ started at $x/L_t(0)$ and killed at 0.
\end{itemize}
We now relate these processes to the original process $X$ using a \emph{many-to-one lemma}. We use the version from Harris and Roberts \cite{HR}; see Section 4.1 in that paper. The first step is to define a branching process with spine through a measure change of the original process using a martingale of the form
\[
\sum_{u \in N(s)} \zeta(u, s) e^{-\beta m s},
\]
where we recall that $\beta$ is the branching rate, $m+1$ is the mean of the offspring distribution and $\beta m = 1/2$ by assumption. 
In our case, we set
\begin{align*}
\zeta(u, s) &= e^{X_u(s) - X_u(0) + s/2},\quad s\ge 0,\ i\le N(s),
\end{align*}
and note that this corresponds to a martingale which transforms the law of a Brownian motion with drift $-1$ to the law of a Brownian motion without drift. With this definition, the law  $\mathbb Q^1_x$ from Harris and Roberts \cite{HR} is precisely the law of our process $X'$. Furthermore, for some positive functional $G$ and for all $u \in N(\delta t)$, set
\begin{align*}
Y(u) &=  F((X(s))_{s\in[0,\delta t]})G((X_u(s))_{s\in[0,\delta t]})e^{X_u(\delta t)} \\
&= F((X(s))_{s\in[0,\delta t]})G((X_u(s))_{s\in[0,\delta t]})e^{X_u(0)}\zeta(u,\delta t) e^{-\delta t/2}.
\end{align*}
By the Many-to-one lemma from \cite[Section 4.1]{HR}, we then have that 
\begin{align}
\label{eq:many_to_one}
\E_x\left[\sum_{u \in N(\delta t)} Y(u)\right]
= e^x \E_x\left[F((X'(s))_{s\in[0,\delta t]})G((X_\xi'(s))_{s\in[0,\delta t]})\right].
\end{align}
We now specialize this to 
\[
G((x(s))_{s\in[0,\delta t]}) = \sin\left(\frac{\pi x(\delta t)}{L_t(\delta t)}\right) \1_{x(s)\in [0,L_t(s)]\,\forall s\le \delta t},
\]
noting that with this definition, we have
\[
Z_t'(\delta t) = L_t(\delta t)e^{-L_t(\delta t)} \sum_{u \in N(\delta t)} e^{X_u(\delta t)}G((X_u(s))_{s\in[0,\delta t]}).
\]
Using \eqref{eq:many_to_one},  we get
\begin{equation}
\label{eq:Xprime}
\E_x\left[F((X(s))_{s\in[0,\delta t]})Z_t'(\delta t)\right] 
= L_t(\delta t)e^{x-L_t(\delta t)} \E_x\left[F((X'(s))_{s\in[0,\delta t]})G((X_\xi'(s))_{s\in[0,\delta t]})\right].
\end{equation}
Conditioning on $X_\xi'$ in the expectation on the right-hand side of \eqref{eq:Xprime} and applying Lemma~\ref{lem:BM_girsanov} with $(X_\xi'(s))_{s\ge0}$ taking the role of $(B_s)_{s\ge0}$ and $(B''(u))_{u \geq 0}$ taking the role of $(b_u)_{u \geq 0}$ gives
\begin{align}
\label{eq:Xprimeprime}
&\E_x\left[F((X'(s))_{s\in[0,\delta t]})G((X_\xi'(s))_{s\in[0,\delta t]})\right] \nonumber \\
&\qquad = e^{O(\delta+t^{-1/3})}\E_x\left[F((X''(s))_{s\in[0,\delta t]})\sin(\pi B''_{\tau(\delta t)})\1_{B''_u \in [0,1]\,\forall u\in [0,\tau(\delta t)]}\right]. 
\end{align}
Finally, recall that the Brownian taboo process is obtained from Brownian motion killed outside the interval $[0,1]$ by a Doob transform using the $-(\pi^2/2)$-harmonic function $\sin(\pi x)$; see \eqref{eq:taboo}. This gives
\begin{align}
\label{eq:Xtilde}
&\E_x\left[F((X''(s))_{s\in[0,\delta t]})\sin(\pi B''_{\tau(\delta t)})\1_{B''_u \in [0,1]\,\forall u\in [0,\tau(\delta t)]}\right] \nonumber \\
&\qquad \qquad = e^{-\pi^2 \tau(\delta t)/2}\sin\bigg( \frac{\pi x}{L_t(0)} \bigg) \E_x\left[F((\tilde X(s))_{s\in[0,\delta t]})\right].
\end{align}
Collecting \eqref{eq:Xprime}, \eqref{eq:Xprimeprime} and \eqref{eq:Xtilde}, and noting that $L_t(\delta t) = L_t(0)e^{O(\delta)}$ and $e^{-\pi^2 \tau(\delta t)/2} = e^{L_t(\delta t)-L_t(0)}$, we get
\[
\E_x\left[F((X(s))_{s\in[0,\delta t]})Z_t'(\delta t)\right]  = e^{O(\delta+t^{-1/3})} z_t(0,x) \E_x\left[F((\tilde X(s))_{s\in[0,\delta t]})\right],
\]
which was to be proven.
\end{proof}

\begin{proof}[Proof of Proposition~\ref{lem:spine_comparison}]
Let $\eps,\eta,\delta >0$. In the course of the proof, we will choose $\eta$ small depending on $\eps$, and then $\delta$ small depending on $\eps$ and $\eta$. Define the event
\[
G = \{Z_t(\delta t) \le \eta,\ M(\delta t) \le L_t(\delta t) - \eta^{-1},\ Z_t(\delta t) = Z_t'(\delta t)\}
\]
By the assumption on the initial condition and results from Maillard and Schweinsberg~\cite{ms20}, we will show below that for every $\eps,\eta>0$ there exists $\delta>0$ such that
\begin{equation}
\label{eq:prob_G_delta}
\P_x(G\,|\,\zeta > t) \ge 1-\eps.
\end{equation}
Furthermore, by the second part of Theorem~1.1 in \cite{ms20}, there exists a constant $\alpha\in(0,\infty)$ such that, choosing $\eta$ small enough, we have for large enough $t$,
\begin{equation}
\label{eq:prob_zeta_ge_t}
\P_x(\zeta > t) = e^{O(\eps)} \alpha z_t(0,x)
\end{equation}
and
\begin{equation}
\label{eq:cond_G_delta}
\P_x(\zeta > t\,|\,\mathcal F_{\delta t})\1_{G} = e^{O(\eps)} \alpha Z_t'(\delta t)\1_{G}.
\end{equation}

We now wrap up the proof, assuming~\eqref{eq:prob_G_delta}. It is enough to show that for every functional $H$ with $0\le H\le 1$, we have for large enough $t$,
\begin{equation}
\label{eq:toshow}
\E_x[H((X(r))_{r\in[0,\delta t]})\,|\,\zeta>t] \le \E_x[H((\tilde X(r))_{r\in[0,\delta t]})] e^{O(\eps+\delta)} + O(\eps+\delta).
\end{equation}
Let $H$ be such a functional. By \eqref{eq:prob_G_delta}, we have
\begin{equation}
\label{eq:exp_H_G_delta}
\E_x[H((X(r))_{r\in[0,\delta t]})\,|\,\zeta>t] = \E_x[H((X(r))_{r\in[0,\delta t]})\1_{G}\,|\,\zeta>t] + O(\eps).
\end{equation}
Furthermore, by \eqref{eq:prob_zeta_ge_t} and \eqref{eq:cond_G_delta}, we have
\begin{align*}
\E_x[H((X(r))_{r\in[0,\delta t]})\1_{G}\,|\,\zeta>t] &= \frac{\E_x[H((X(r))_{r \in [0, \delta t]}) \1_G \P_x(\zeta > t \,|\, \mathcal{F}_{\delta t})]}{\P_x(\zeta > t)} \\
&= \frac{e^{O(\eps)} \E_x[H((X(r))_{r \in [0, \delta t]} \1_G Z_t'(\delta t)]}{z_t(0, x)} \\
&\le \frac{e^{O(\eps)}\E_x[H((X(r))_{r\in[0,\delta t]})Z_t'(\delta t)]}{z_t(0,x)},
\end{align*}
and Lemma~\ref{lem:spine_Z} immediately yields for large enough $t$,
\begin{equation}
\label{eq:almost_done}
\E_x[H((X(r))_{r\in[0,\delta t]})\1_{G}\,|\,\zeta>t] \le e^{O(\eps+\delta)} \E_x[H((\tilde X(r))_{r\in[0,\delta t]})].
\end{equation}
Equations~\eqref{eq:exp_H_G_delta} and \eqref{eq:almost_done} now yield \eqref{eq:toshow}, which finishes the proof of the lemma, assuming \eqref{eq:prob_G_delta}.

We finish by proving \eqref{eq:prob_G_delta}. We bound the probability of the complement in three steps. Recall that we assume that $L_t-x\to\infty$ as $t\to\infty$, which in particular implies that $z_t(0,x)\to 0$ as $t\to\infty$. Theorem~2.4 in \cite{ms20} then says that the finite-dimensional distributions of the process $(Z_t(st))_{s\in [0,1)}$, conditioned on $\zeta > t$, converge, as $t\to\infty$, to those of a certain stochastically continuous process starting from 0. This entails that for every $\eps,\eta>0$, one can choose $\delta$ small enough such that for large enough $t$,
\[
\P_x(Z_t(\delta t) > \eta\,|\,\zeta>t) \le \eps.
\]
Furthermore, Theorem~2.9 in \cite{ms20} yields that for large enough $t$,
\[
\P_x(M(\delta t) > L_t(\delta t) - \eta^{-1}\,|\,\zeta>t) \le \eps.
\]
Finally, Lemma~5.8 in \cite{ms20} applied with $r=A=0$ and $s=\delta t$ gives that for large enough $t$,
\[
\P_x(M(s) > L_t(s)\,\text{for some $s\in[0,\delta t]$}) = O(\delta z_t(0,x)).
\]
Using \eqref{eq:prob_zeta_ge_t} and using that $Z_t(\delta t) = Z_t'(\delta t)$ on the event that $M(s) \le L_t(s)$ for all $s\in[0,\delta t]$, this implies that
\[
\P_x(Z_t(\delta t) \ne Z_t'(\delta t)\,|\,\zeta > t) = O(\delta).
\]
It now follows that one can choose $\delta$ small enough such that for large enough $t$, the result \eqref{eq:prob_G_delta} holds. This finishes the proof.
\end{proof}

\section{Excursions of the Brownian taboo process}\label{taboosec}

The main goal of this section is to prove the following lemma, which generalizes (\ref{Rayleighsimple}) in the introduction and is the key to the proof of Theorem \ref{th:global_max}.

\begin{Lemma}
\label{lem:taboo_with_drift}
Suppose that for every $\gamma>0$, we have a Brownian taboo process $(K_\gamma(s))_{s\ge0}$ starting a (possibly random) point $z_\gamma \in [0,1]$. Assume that we have the following limit in probability:
\begin{equation}\label{initialz}
\lim_{\gamma \rightarrow 0} \frac{z_{\gamma}}{\sqrt{\gamma}} = \infty.
\end{equation}
Suppose furthermore we have a (possibly random) function $g: \R^+ \rightarrow \R^+$ such that, in probability,
\begin{equation}\label{gcond}
\lim_{\gamma \rightarrow 0} \gamma^{1/2} g(\gamma) = \infty,
\end{equation}
and for all $\gamma > 0$, we have (possibly random) continuous functions $b_{\gamma}, d_{\gamma}: [0, g(\gamma)] \rightarrow [0, \infty)$ such that the following limits hold in probability:
\begin{equation}\label{adhyp}
\lim_{\gamma \rightarrow 0} \sup_{0 \leq s \leq g(\gamma)} |1 - b_{\gamma}(s)| = 0, \qquad \lim_{\gamma \rightarrow 0} \sup_{0 < s \leq g(\gamma)} \bigg|1 - \frac{d_{\gamma}(s)}{\gamma s} \bigg| = 0.
\end{equation}
Define for $\gamma > 0$,
\begin{equation}\label{Mmdef}
M_\gamma = \min_{0 \leq s \leq g(\gamma)} \, \big(b_{\gamma}(s) K_s + d_{\gamma}(s)\big),\qquad m_\gamma = \operatorname*{arg\,min}_{0 \leq s \leq g(\gamma)} \, \big(b_{\gamma}(s) K_s + d_{\gamma}(s) \big).
\end{equation}
Then, as $\gamma\to 0$, we have the convergence in law
\begin{equation}\label{Rayleighconv}
\bigg(\frac{M_\gamma}{\sqrt \gamma}, \sqrt \gamma m_\gamma \bigg) \Rightarrow ({\tilde R},U {\tilde R}),
\end{equation}
where ${\tilde R}$ and $U$ are independent random variables, ${\tilde R}$ has a Rayleigh distribution with density $\pi^2 r e^{-\pi^2r^2/2}$ on $\R_+$, and $U$ has a uniform distribution on $[0,1]$.  Furthermore, define for $\theta > 0$,
$$M_{\gamma}^*(\theta) = \min_{\substack{0 \leq s \leq g(\gamma) \\ |s - m_{\gamma}| \geq \theta/\sqrt{\gamma}}}
\big(b_{\gamma}(s) K_s + d_{\gamma}(s)\big).$$ 
Then, for every $\kappa>0$, there exists $\eta > 0$, such that for every $\theta > 0$,
\begin{equation}\label{secondexc}
\limsup_{\gamma \rightarrow 0} P(M_{\gamma}^*(\theta) - M_{\gamma} \leq \eta \sqrt{\gamma}) < \kappa.
\end{equation}
\end{Lemma}

The main tool used in the proof of Lemma~\ref{lem:taboo_with_drift} is the excursion theory of the Brownian taboo process due to Lambert \cite{lambert}, who also extended the results to more general completely asymmetric L\'evy processes. We review this theory here.

Let $(K_s)_{s \geq 0}$ be a Brownian taboo process.  For $0 < x < 1$, let
\begin{equation}\label{px}
p(x) = 2 \sin^2 (\pi x),
\end{equation}
so that $p(x)$ is the density of the speed measure $m(dx)$ and therefore is the stationary density for the process; see Theorem 3.1 of \cite{lambert}.  Denote the local time at $x \in (0, 1)$ by $$\ell_s^{(x)} = \lim_{\eps \rightarrow 0+} \frac{1}{2 \eps} \int_0^s \1_{\{|K_r - x| < \eps\}} \: dr.$$  Corollary 4.3 of \cite{lambert} states that
\begin{equation}\label{Lsx}
\lim_{s \rightarrow \infty} \frac{\ell_s^{(x)}}{s} = p(x) \hspace{.2in}\mbox{a.s.}
\end{equation}
Let $\sigma_u^x = \inf\{s > 0: \ell^{(x)}_s > u\}$ be the (right-continuous) inverse local time.  If $\sigma_{u-}^x < \sigma_u^x$, then let $e_u$ be the excursion defined by $e_u(r) = K_{\sigma_{u-}^x + r}$ for $0 \leq r \leq \sigma_{u-}^x - \sigma_u^x$, and $e_u(r) = x$ for all $r > \sigma_{u-}^x - \sigma_u^x$.  Let $\mathcal{U}$ denote the set of all continuous functions $f: [0, \infty) \rightarrow (0, 1)$ such that for some $y > 0$, we have $f(0) = f(y) = x$, $f(r) \neq x$ for all $r \in (0, y)$, and $f(r) = x$ for all $r > y$.  We interpret $\mathcal{U}$ as the space of possible excursions, with $y$ corresponding to the excursion length.  Then the set of points $(u, e_u)$ is a Poisson point process on $\R^+ \times \mathcal{U}$ whose intensity measure is given by the product of Lebesgue measure and an excursion measure on $\mathcal{U}$ which, following \cite{lambert}, we denote by $n_x$.  For $0 < d < 1 - x$, let $$A_d = \Big\{f \in \mathcal{U}: \sup_{r \geq 0} f(r) > 1 - d\Big\}$$ be the set of excursions whose maximum exceeds $1 - d$.  Proposition 4.2 of \cite{lambert}, applied with $1 - d - x$ in place of $\eta$ and using the values of $\rho$ and $W^{-(\rho)}(x)$ from the top of p. 256 in \cite{lambert}, states that
$$n_x(A_d) = \frac{\pi}{2} \cdot \frac{\sin (\pi(1-d))}{\sin (\pi(1-d-x)) \sin (\pi x)}.$$

We now consider the case $x = 1/2$.  By symmetry, if $0 < d < 1/2$ and we define
$$A_d^* = \Big\{f \in \mathcal{U}: \inf_{r \geq 0} f(r) < d\Big\},$$
then
\begin{equation}\label{Adstar}
n_{1/2}(A_d^*) = n_{1/2}(A_d) = \frac{\pi}{2} \cdot \frac{\sin(\pi d)}{\sin(\frac{\pi}{2} - \pi d)}.
\end{equation}
For the excursion $e_u$, write $a_u = \inf_{r \geq 0} e_u(r)$.  Let $\mathcal N$ be the point process on $\R^+ \times (0, 1/2)$ consisting of the points $(u, a_u)$ for which $a_u < 1/2$, so we are recording here only excursions below $1/2$.  Then $\mathcal N$ is a Poisson point process whose intensity measure is given by the product of Lebesgue measure and a measure $\nu$ such that $\nu((0, d)) = n_{1/2}(A_d^*)$ for $0 < d < 1/2$.
By differentiating the right-hand side of (\ref{Adstar}), we see that the measure $\nu$ has a density $h$, which is positive on $(0, 1/2)$ and satisfies
\begin{equation}\label{hdlimit}
\lim_{d \rightarrow 0} h(d) = \frac{\pi^2}{2}.
\end{equation}
Also, we write $\ell_s = \ell_s^{(1/2)}$ and $\sigma_u = \sigma_u^{1/2}$, and then equations (\ref{px}) and (\ref{Lsx}) imply
\begin{equation}\label{Ls12}
\lim_{s \rightarrow \infty} \frac{\ell_s}{s} = 2.
\end{equation}

\def\MN{N}
\def\mN{u^*}

In order to prepare the proof of Lemma~\ref{lem:taboo_with_drift}, we first prove a similar result regarding the Poisson process $\mathcal N$. 
Denote by $\mathcal N^{(u)}$ the projection of $\mathcal N$, seen as a set, onto the first coordinate, i.e. it is the (random) set of $u \in \R_+$, such that $(u,a_u)\in \mathcal N$ for some $a_u$.
Define
\[
\MN_\gamma = \min_{u\in \mathcal N^{(u)}} \left\{a_u + \gamma \frac u 2\right\},\quad \mN_\gamma = \operatorname*{argmin}_{u\in \mathcal N^{(u)}} \left\{a_u + \gamma \frac u 2\right\},\quad \MN_\gamma^* =  \min_{u\in \mathcal N^{(u)},\ u\ne \mN_\gamma} \left\{a_u + \gamma \frac u 2\right\}.
\]
Note that the minimum in the definition of $\MN_\gamma$ is attained almost surely and at a unique point, so that $\mN_\gamma$ is well-defined. To see this, first define $\MN_\gamma$ with the minimum replaced by an infimum, and note that $\MN_\gamma < 1/2$ almost surely, because the measure $\nu$ gives infinite mass to every interval $[a,1/2)$, for $a<1/2$. Now, for every $A\in (0,1/2)$, on the event $\{\MN_\gamma < A\}$, every sequence of points $(u_n,a_{u_n})\in \mathcal N$ approaching this infimum has to  be contained in the triangle $\{(u,a)\in \R_+^2: a+\gamma u/2 \le A\}$  for large enough $n$. But the number of points in this triangle is finite almost surely, and so $u_n$ is a minimizer for large $n$. The uniqueness of the minimizer comes from the continuity of the intensity measure of the point process $\mathcal N$.

We now have the following result:
\begin{Lemma}
\label{lem:Poisson}
Define $\tilde R$ and $U$ as in the statement of Lemma~\ref{lem:taboo_with_drift}.  Then as $\gamma \rightarrow 0$, we have
\[
\left(\frac{1}{\sqrt\gamma}\MN_\gamma,\sqrt\gamma \mN_\gamma\right) \Rightarrow (\tilde R,2U\tilde R),
\]
and
\[
\frac{1}{\sqrt\gamma} \left(\MN_\gamma^* - \MN_\gamma\right) \Rightarrow X^*,
\]
where $X^*$ is some strictly positive r.v. Moreover, the statement still holds if, in the definitions of $\MN_\gamma$, $\mN_\gamma$ and $\MN^*_\gamma$ one adds an additional constraint $u \in [g_1(\gamma),g_2(\gamma)]$ for some functions $g_1,g_2$ satisfying $0\le g_1(\gamma) \ll \gamma^{-1/2} \ll g_2(\gamma) \le +\infty$.
\end{Lemma}
\begin{proof}
We first consider the case $g_1 \equiv 0$ and $g_2 \equiv +\infty$.
Define the point process $\mathcal N_\gamma$ by rescaling the first coordinate of $\mathcal N$ by $\sqrt \gamma$ and the second by $1/\sqrt \gamma$, i.e. $\mathcal N_\gamma$ consists of the points $(v,a^\gamma_v) = (u\sqrt\gamma,a_u/\sqrt\gamma)$ for $(u,a_u)\in \mathcal N$. It follows that $\mathcal N_\gamma$ is a Poisson point process with intensity measure $\nu_\gamma \otimes du$, where
\[
\nu_\gamma = \nu(\sqrt\gamma\,\cdot )/\sqrt \gamma.
\]
We denote by $\mathcal N_\gamma^{(v)}$ the projection of $\mathcal N_\gamma$ onto the first coordinate. Then, we have
\begin{align*}
\MN_\gamma &= \sqrt\gamma \left(\min_{v\in \mathcal N_\gamma^{(v)}} \left\{ a^\gamma_v + \frac{v}{2} \right\} \right),\\
\mN_\gamma &= (\sqrt\gamma)^{-1} v_\gamma^*,\quad\text{where } v_\gamma^* := \operatorname*{argmin}_{v \in \mathcal N_\gamma^{(v)}} \left\{ a^\gamma_v+ \frac{v}{2} \right\},\\
\MN_\gamma^* &= \sqrt\gamma \left(\min_{v\in \mathcal N_\gamma^{(v)},\ v\ne v_\gamma^*} \left\{ a^\gamma_v + \frac{v}{2} \right\} \right).
\end{align*}
Note that $v_\gamma^* \le 2  \min_{v\in \mathcal N_\gamma^{(v)}} a^\gamma_v + \frac{v}{2}$.
Hence, it suffices to show that
\begin{equation}
\label{eq:poisson_toshow1}
\left(\min_{v\in \mathcal N_\gamma^{(v)}} a^\gamma_v + \frac{v}{2}, v_\gamma^*\right) \Rightarrow (\tilde R,2U\tilde R),
\end{equation}
and
\begin{equation}
\label{eq:poisson_toshow2}
\left(\min_{v\in \mathcal N_\gamma^{(v)},\ v\ne v_\gamma^*} a^\gamma_v + \frac{v}{2}\right) - \left(\min_{v\in \mathcal N_\gamma^{(v)}} a^\gamma_v + \frac{v}{2}\right) \Rightarrow X^*,
\end{equation}
with $X^*$ as above.
By \eqref{hdlimit}, one sees that $\nu_\gamma$ vaguely converges to $\pi^2/2$ times Lebesgue measure on $\R_+$; in fact, its density converges locally uniformly. By standard thinning and superposition arguments for Poisson processes, it follows that one can couple $\mathcal N_\gamma$ with a Poisson process $\mathcal N_0$ with intensity measure $\pi^2/2$ times Lebesgue measure on $\R_+^2$ in such a way that on every compact set, they are equal with probability going to 1 as $\gamma \to 0$. Using similar arguments to the ones used to show that $\min_{v\in \mathcal N_\gamma^{(v)}} a^\gamma_v + \frac{v}{2}$ is attained almost surely and denoting the points of $\mathcal N_0$ by $(v, a_v^0)$, it follows that the following convergence in law holds:
\begin{equation}
\label{eq:poisson_convergence}
\left(\min_{v\in \mathcal N_\gamma^{(v)}} \left\{ a^\gamma_v + \frac{v}{2} \right\},\ v_\gamma^*,\ \min_{v\in \mathcal N_\gamma^{(v)},\ v\ne v_\gamma^*} \left\{ a^\gamma_v + \frac{v}{2} \right\} \right) \Rightarrow \left(M_0,\ v_0^*,\ M_0^*\right),
\end{equation}
where $M_0 = \min_{v\in \mathcal N_0^{(v)}} a^0_v + \frac{v}{2}$, $v_0^* = \operatorname*{argmin}_{v\in \mathcal N_0^{(v)}} a^0_v + \frac{v}{2}$ and $M_0^* = \min_{v\in \mathcal N_0^{(v)},\ v\ne v_0^*} a^0_v + \frac{v}{2}$.

For $t>0$, let $C_t = \{(v,a)\in \R_+^2, a+v/2\le t\}$, which is a right triangle whose two short sides have lengths $t$ and $2t$, respectively. Hence, it has Lebesgue measure $t^2$. We therefore have, 
\[
\P(M_0 > t) = \P(\mathcal N_0(C_t) = 0) = e^{-(\pi^2/2) t^2},
\]
so that $M_0$ is equal in law to $\tilde R$. Moreover, conditioning on $M_0 = t$, we have that $v_0^*$ is uniformly distributed on $[0,2t]$ (and $a^0_{v_0^*} = t - v_0^*/2$).  Hence, we have
\[
\left(M_0 , v_0^*\right) \stackrel{\mathrm{law}}= (\tilde R,2U\tilde R).
\]
Furthermore, the remaining points form again a Poisson process with the same intensity measure but restricted to $\R_+^2\backslash C_t$. In particular, using the continuity of the intensity measure,
\[
X^* := M_0^* - M_0 > 0,\quad \text{almost surely.}
\]
These two facts, together with \eqref{eq:poisson_convergence}, yield \eqref{eq:poisson_toshow1} and \eqref{eq:poisson_toshow2} and finish the proof in the case $g_1 \equiv 0$ and $g_2 \equiv +\infty$.

To cover the general case, note that since $\sqrt{\gamma}\mN_\gamma$ converges in law to a positive, finite random variable as $\gamma\to 0$, the assumptions imply that $\mN_\gamma\in [g_1(\gamma),g_2(\gamma)]$ with high probability. This proves the first statement in the general case. For the second statement, it follows from the above proof that the minimizer for $\MN_\gamma^*$ is again of order $\gamma^{-1/2}$, so that it is in the interval $[g_1(\gamma),g_2(\gamma)]$ as well with high probability. This concludes the proof.
\end{proof}

\begin{proof}[Proof of Lemma \ref{lem:taboo_with_drift}]
We wish to reduce the problem so that we can apply Lemma~\ref{lem:Poisson}. We do this in several steps.

\underline{Step 1.}
We first argue that it is sufficient to prove (\ref{Rayleighconv}) and (\ref{secondexc}) when $b_{\gamma}(s) = 1$ and $d_{\gamma}(s) = \gamma s$ for all $s \geq 0$.  To see this, define
\begin{equation}\label{Mtildedef}
{\tilde M}_\gamma = \min_{0 \leq s \leq g(\gamma)} \, \big(K_s + \gamma s \big),\quad {\tilde m}_\gamma = \operatorname*{argmin}_{0 \leq s \leq g(\gamma)} \, \big(K_s + \gamma s \big), \quad {\tilde M}_{\gamma}^*(\theta) = \min_{\substack{0 \leq s \leq g(\gamma) \\ |s - {\tilde m}_{\gamma}| \geq \theta/\sqrt{\gamma}}}
\big(K_s + \gamma s \big).
\end{equation}
Suppose we can show that 
\begin{equation}\label{Rayleightilde}
\bigg(\frac{{\tilde M}_\gamma}{\sqrt \gamma}, \sqrt \gamma {\tilde m}_\gamma \bigg) \Rightarrow ({\tilde R},U {\tilde R}),
\end{equation}
and that for every $\kappa>0$, there exists $\eta > 0$, such that for every $\theta > 0$, we have
\begin{equation}\label{secondexctilde}
\limsup_{\gamma \rightarrow 0} P({\tilde M}_{\gamma}^*(\theta) - {\tilde M}_{\gamma} \leq \eta \sqrt{\gamma}) < \kappa.
\end{equation}
For $s \geq 0$, we have $$\bigg|\frac{b_{\gamma}(s) K_s + d_{\gamma}(s)}{K_s + \gamma s} - 1 \bigg| = \bigg| \frac{(b_{\gamma}(s) - 1)K_s + (\frac{d_{\gamma}(s)}{\gamma s} - 1) \gamma s}{K_s + \gamma s} \bigg| \leq |b_{\gamma}(s) - 1| + \bigg|\frac{d_{\gamma}(s)}{\gamma s} - 1 \bigg|.$$
It now follows from (\ref{adhyp}) that
\begin{equation}\label{ratiobound}
\lim_{\gamma \rightarrow 0} \sup_{0 \leq s \leq g(\gamma)} \bigg|\frac{b_{\gamma}(s) K_s + d_{\gamma}(s)}{K_s + \gamma s} - 1 \bigg| = 0.
\end{equation}
Therefore, we have $M_{\gamma}/{\tilde M}_{\gamma} \rightarrow_p 1$ as $\gamma \rightarrow 0$.  Also, (\ref{Rayleightilde}), (\ref{secondexctilde}) and (\ref{ratiobound}) imply that for all $\theta > 0$ and $\kappa > 0$, we have
$$\limsup_{\gamma \rightarrow 0} P \bigg( |m_{\gamma} - {\tilde m}_{\gamma}| > \frac{\theta}{\sqrt{\gamma}} \bigg) < \kappa,$$
which implies that $\sqrt{\gamma} |m_{\gamma} - {\tilde m}_{\gamma}| \rightarrow_p 0$ as $\gamma \rightarrow 0$.  Therefore, the convergence (\ref{Rayleightilde}) implies (\ref{Rayleighconv}).  Also, by (\ref{ratiobound}) and the result $\sqrt{\gamma} |m_{\gamma} - {\tilde m}_{\gamma}| \rightarrow_p 0$, the convergence (\ref{secondexctilde}) implies (\ref{secondexc}).  We will therefore aim to prove (\ref{Rayleightilde}) and (\ref{secondexctilde}).

\underline{Step 2.}
Next, we note that additionally, we can consider the process over an infinite time horizon.  For $s\ge g(\gamma)$, we have 
\[
K_s + \gamma s \ge \gamma s \ge \gamma g(\gamma) \gg \sqrt{\gamma},
\]
and so, since $M_\gamma$ is of order $\sqrt{\gamma}$ in probability as $\gamma \to 0$, the minimum is attained in the time interval $[0,g(\gamma)]$ with high probability. 
Therefore, it is enough to prove the result with $g(\gamma) = +\infty$.

\underline{Step 3.}
We recall from the two previous steps that we can and will assume $g(\gamma) = +\infty$, $b_{\gamma}(s) = 1$ and $d_{\gamma}(s) = \gamma s$ for all $s \geq 0$. We now show that we can ignore a certain time interval at the beginning. Recall that $\sigma_u$ denotes the inverse local time at $1/2$ of the Brownian taboo process. Set
\begin{align*}
M^{(1)}_\gamma &= \min_{0 \leq s \leq \sigma_{\gamma^{-1/4}}} \, \big(K_s + \gamma s \big).
\end{align*}
We claim that
\begin{equation}
\label{eq:M1ggsqrtgamma}
M^{(1)}_\gamma/\sqrt{\gamma} \rightarrow_p \infty \quad \text{as $\gamma \to 0$.}
\end{equation}
For this, we first bound from below
\begin{align}
\label{eq:M1decomposition}
M^{(1)}_\gamma \ge \min_{0 \leq s \leq \sigma_{\gamma^{-1/4}}} \, K_s = \min\left(\min_{s \leq \sigma_0}  \, K_s,\  \min_{u\in \mathcal N^{(u)},\,u\le \gamma^{-1/4}} \, a^\gamma_u\right).
\end{align}
We now consider separately both terms on the RHS of \eqref{eq:M1decomposition}.

\textit{Step 3a.} We start with the term $\min_{s \leq \sigma_0}  \, K_s$. If $z_\gamma \ge 1/2$, then this term equals 1/2. 
 If $z_{\gamma} < 1/2$, then by the strong Markov property of the Brownian taboo process, the probability that it is smaller than some $z<z_\gamma$ is the probability that for the Brownian taboo process started at $1/2$, the first excursion of the process below $z_{\gamma}$ also goes below $z$. This equals\footnote{One can also calculate this using the scale function of the Brownian taboo process.}
 $$\frac{n_{1/2}\big(A^*_{z}\big)}{n_{1/2}(A^*_{z_{\gamma}})} = \frac{\nu((0, z))}{\nu((0, z_{\gamma}))}.$$ From (\ref{hdlimit}), it follows that $\min_{s \leq \sigma_0}  \, K_s$ is of the same order as $z_\gamma$, with high probability as $\gamma\to 0$. From (\ref{initialz}), it then follows that $(\min_{s \leq \sigma_0} \, K_s)/\sqrt{\gamma} \rightarrow_p \infty$ as $\gamma\to0$.
 
\textit{Step 3b.} 
For the second term on the RHS of \eqref{eq:M1decomposition}, a first moment argument shows that,
\[
\P\left(\min_{u\in \mathcal N^{(u)},\,u\le \gamma^{-1/4}} \, a^\gamma_u \le \gamma^{3/8}\right) \le \gamma^{-1/4} \nu((0,\gamma^{3/8})),
\]
and \eqref{hdlimit} shows that this goes to 0 as $\gamma\to 0$. This shows that $(\min_{u\in \mathcal N^{(u)},\,u\le \gamma^{-1/4}} \, a^\gamma_u)/\sqrt{\gamma} \rightarrow_p \infty$ as $\gamma\to 0$.

Steps 3a and 3b combined with \eqref{eq:M1decomposition} now yield \eqref{eq:M1ggsqrtgamma}.

\underline{Step 4.}
Taking into account the last step, it remains to show that the statement of the lemma still holds if we add the condition $s\ge \sigma_{\gamma^{-1/4}}$ to all quantities (and still assuming $g(\gamma) = +\infty$, $b_{\gamma}(s) = 1$ and $d_{\gamma}(s) = \gamma s$ for all $s \geq 0$).

Recall from \eqref{Ls12} that $\ell_s/s \to 2$ almost surely as $s\to\infty$. In fact, the speed in this convergence is uniformly bounded in the starting point $z_\gamma\in[0,1]$. Indeed, note that the time for the Brownian taboo process to reach $1/2$ from any starting point is stochastically bounded by the time for the process to reach $1/2$ started from $0$ or $1$, which is a finite random variable because $0$ and $1$ are entrance boundaries for the taboo process. The same holds for the inverse local time. Hence, there exists $\eta(\gamma) \to 0$ as $\gamma\to0$, such that,
\begin{align}
\label{eq:uniform_ell_s}
&\forall s\ge \sigma_{\gamma^{-1/4}}: 2(1-\eta(\gamma)) \le \frac{\ell_s}{s} \le 2 (1+\eta(\gamma)), \quad \text{with high probability as $\gamma\to 0$},\\
\label{eq:uniform_sigma_u}
&\forall u\ge \gamma^{-1/4}: \frac{1-\eta(\gamma)}{2} \le \frac{\sigma_{u-}}{u} \le \frac{\sigma_u}{u} \le \frac{1+\eta(\gamma)}{2}, \quad \text{with high probability as $\gamma\to 0$}.
\end{align}
With the same arguments as in Step 1, it follows from \eqref{eq:uniform_ell_s} that it is enough to show the result with $d_\gamma(s) = \gamma \ell_s/2$. Now note that since $\ell_s$ is constant on the open time interval delimiting an excursion, we have the following:
\begin{align}
\label{eq:from_s_to_u}
\min_{s\ge \sigma_{\gamma^{-1/4}}} \, \big(K_s + \gamma \ell_s/2 \big) = \min_{u\in \mathcal N^{(u)}, u>\gamma^{-1/4}} \, \big(a^\gamma_u + \gamma u/2 \big).
\end{align}
Moreover, the minimizers $s^*$ and $u^*$ satisfy $s^* \in [\sigma_{u^*-},\sigma_{u^*}]$. The statement now readily follows from Lemma~\ref{lem:Poisson} (with $g_1(\gamma) = \gamma^{-1/4}$ and $g_2(\gamma) = +\infty$), and \eqref{eq:uniform_sigma_u}.
\end{proof}

\section{The all-time maximum}\label{maxsec}

In this section, we prove Theorem~\ref{th:global_max}.  We first recall that if $0 \leq s \leq t$, then $$L_t(s) = c(t - s)^{1/3}, \qquad \tau(s) = \int_0^s \frac{1}{L_t(u)^2} \: du.$$  
We next recall two lemmas about the extinction times for branching Brownian motion.  The result (\ref{survivelim}) below is an immediate consequence of Theorem 1.3 in \cite{ms20}; note that the function $\phi(v)$ in (\ref{survivelim}) corresponds to $1 - \phi(-cv/3)$ in Theorem 1.3 of \cite{ms20}.  The result (\ref{survive2}) is part of Theorem~1 of \cite{bbs14} and is stated in the scaling of the present paper in (1.4) of \cite{ms20}.

\begin{Lemma}\label{survivalms}
Let $q$ be the extinction probability for a Galton-Watson process with offspring distribution $(p_k)_{k=1}^{\infty}$, and let $v \in \R$.  If $x = L_t$, then
\begin{equation}\label{survivelim}
\lim_{t \rightarrow \infty} \P_x(\zeta > t + vt^{2/3}) = \phi(v),
\end{equation}
where $\phi$ is a decreasing function satisfying $\lim_{z \rightarrow \infty} \phi(z) = 0$ and $\lim_{z \rightarrow -\infty} \phi(z) = 1-q.$  Also, there is a positive constant $C$ such that if $0 < x \leq L_t - 1$, then
\begin{equation}\label{survive2}
\P_x(\zeta > t) \leq C L_t \sin \bigg( \frac{\pi x}{L_t} \bigg) e^{x - L_t}. 
\end{equation}
\end{Lemma}

\noindent Lemma \ref{survival23} is a special case of Proposition 1.4 in \cite{ms20}.

\begin{Lemma}\label{survival23}
Suppose $x > 0$, possibly depending on $t$, and $\lim_{t \rightarrow \infty} (L_t - x) = \infty$.  Then, for all $v > 0$, we have
$$\lim_{t \rightarrow \infty} \P_x(\zeta > t + vt^{2/3} \,|\, \zeta > t) = e^{-cv/3}.$$
\end{Lemma}

We next prove three lemmas concerning the maximum position for branching Brownian motion with absorption started from one particle at $x$.  Lemma \ref{MGbound} controls the all-time maximum position for branching Brownian motion with absorption, while Lemmas \ref{rtmostms} and \ref{rightLD} bound the maximum position that a particle achieves after a certain time.  Note that in Lemmas \ref{rtmostms} and \ref{rightLD}, the position $x$ of the initial particle is allowed to depend on $t$.

\begin{Lemma}\label{MGbound}
Let $x > 0$, possibly depending on $t$, and let $a > 0$.  Then
$$\P_x \Big( \sup_{s \geq 0} M(s) \geq x + a \Big) \leq e^{-a}.$$
\end{Lemma}

\begin{proof}
For $s \geq 0$, let $$V(s) = \sum_{u \in N(s)} X_u(s) e^{X_u(s)}.$$  Then the process $(V(s), s \geq 0)$ is a nonnegative martingale, as shown, for example, in Lemma 2 of \cite{hh07}.  Let $\sigma = \inf\{s: M(s) \geq x+a\}$.  Then, on the event $\sigma < \infty$, we have $V(\sigma) \geq (x+a)e^{(x+a)}$.  Therefore, by the Optional Sampling Theorem, $$\P_x(\sigma < \infty) \leq \frac{xe^{x}}{(x+a)e^{(x+a)}} \leq e^{-a},$$ as claimed.
\end{proof}

\begin{Lemma}\label{rtmostms}
Suppose $x > 0$, and define $t$ such that $x = L_t$.  Then for all $u > 0$ such that $L_{t+u} \ge x+ 1$, we have
$$\P_x \bigg( \sup_{s \geq u} M(s) \geq x \bigg) \leq C_0 (L_{t+u} - x)e^{-(L_{t+u} - x)},$$ where $C_0$ is a positive constant not depending on $t$ or $u$. 
\end{Lemma}

\begin{proof}
We first suppose that $x = L_t \ge 1$. Then, from equation (\ref{survivelim}) with $v = 0$, one can see that there is a positive constant $C_1$ such that $\P_x(\zeta > t) \geq C_1$ for all $x \geq 1$. Let $\sigma = \inf\{s \geq u: M(s) \geq x\}$.  By applying the strong Markov property at time $\sigma$, we get
\begin{align*}
\P_x(\zeta > t + u) &\geq \P_x(\sigma < \infty) \P_x(\zeta > t + u \,|\, \sigma < \infty) \\
&\geq \P_x(\sigma < \infty) \P_x(\zeta > t) \\
&\geq C_1 \P_x(\sigma < \infty).
\end{align*}
Rearranging this equation and applying (\ref{survive2}) with $t+u$ in place of $t$ gives 
\begin{align*}
\P_x(\sigma < \infty) &\leq \frac{1}{C_1} \P_x(\zeta > t + u) \\
&\leq \frac{C}{C_1} L_{t+u} \sin \bigg( \frac{\pi x}{L_{t+u}} \bigg) e^{x - L_{t+u}} \\
&\leq \frac{C \pi}{C_1} (L_{t+u} - x) e^{x - L_{t+u}},
\end{align*}
which implies the statement of the lemma with $C_0 = C \pi/C_1$.

It remains to consider the case $x<1$. Then, translating the whole process by $1-x$ but still keeping the absorption at 0 can only increase the maximal displacement.  Hence,
\[
\P_x \bigg( \sup_{s \geq u} M(s) \geq x \bigg) \le \P_1 \bigg( \sup_{s \geq u} M(s) \geq 1 \bigg).
\]
Now, if $1+x\le L_{t+u} \le 2$, then $(L_{t+u}-x)e^{-L_{t+u}-x} \ge c_0>0$ for some constant $c_0$, and so the statement of the lemma follows with $C_0 = 1/c_0$. On the other hand, if $L_{t+u} > 2$, then if $t_1\ge t$ is such that $L_{t_1} = 1$ and $u_1 = u + t-t_1 \le u$, then we can apply the lemma with $x = 1 = L_{t_1}$ and $u = u_1$ in order to get (noting that $t_1+u_1 = t+u$)
\begin{align*}
\P_1 \bigg( \sup_{s \geq u} M(s) \geq 1 \bigg)
&\le \P_1 \bigg( \sup_{s \geq u_1} M(s) \geq 1 \bigg) \\
&\le C_0(L_{t+u}-1)e^{-(L_{t+u}-1)} \\
&\le e^{1-x}C_0(L_{t+u}-x)e^{-(L_{t+u}-x)}.
\end{align*}
This finishes the proof of the lemma.
\end{proof}

\begin{Lemma}\label{rightLD}
Suppose $x > 0$, possibly depending on $t$, and $\lim_{t \rightarrow \infty} (L_t - x) = \infty$.  Choose positive constants $a$ and $b$ such that $0 < a + 2/3 < b < 1$.  Then
$$\lim_{t \rightarrow \infty} \P_x \bigg( \sup_{s \geq t^b} M(s) \geq L_t - t^a \, \Big|\, \zeta > t \bigg) = 0.$$
\end{Lemma}

\begin{proof}
Let $y = L_t - t^a$.  Note that $y = L_u$, where $$u = \bigg( \frac{y}{c} \bigg)^3 = \bigg(t^{1/3} - \frac{t^a}{c} \bigg)^3 \geq t - \frac{t^a}{c} \cdot 3 t^{2/3},$$ which is greater than $t - t^b/2$ for sufficiently large $t$ because $a + 2/3 < b$.  Therefore, by Lemma~\ref{survivalms}, 
\begin{equation}\label{liminf0}
\liminf _{t \rightarrow \infty} \P_y\bigg(\zeta > t - \frac{t^b}{2} \bigg) \geq \liminf_{t \rightarrow \infty} \P_y (\zeta > u) = \phi(0) > 0.
\end{equation}
Thus, by the strong Markov property applied at the time $\sigma = \inf\{s \geq t^b: M(s) > y\}$, we have
\begin{align}
\P_x\bigg(\zeta > t + \frac{t^b}{2} \, \Big|\, \zeta > t \bigg) &\geq \P_x \bigg( \sup_{s \geq t^b} M(s) \geq y \, \Big|\, \zeta > t \bigg) \P_x \bigg( \zeta > t + \frac{t^b}{2} \,\Big| \, \zeta > t, \: \sup_{s \geq t^b} M(s) \geq y \bigg) \nonumber \\
&\geq \P_x \bigg( \sup_{s \geq t^b} M(s) \geq y \, \Big|\, \zeta > t \bigg) \P_x \bigg( \zeta > t + \frac{t^b}{2} \,\Big| \, \sup_{s \geq t^b} M(s) \geq y \bigg) \nonumber \\
&\geq \P_x \bigg( \sup_{s \geq t^b} M(s) \geq y \, \Big|\, \zeta > t \bigg) \P_y\bigg( \zeta > t - \frac{t^b}{2} \bigg). \label{zetaab}
\end{align}
Because $b > 2/3$, the term on the left-hand side of (\ref{zetaab}) tends to zero as $t \rightarrow \infty$ by Lemma \ref{survival23}.  By (\ref{liminf0}), the second of the two factors on the right-hand side of (\ref{zetaab}) stays bounded away from zero as $t \rightarrow \infty$.  It follows that the first factor on the right-hand side of (\ref{zetaab}) must tend to zero as $t \rightarrow \infty$, which is precisely the statement of the lemma.
\end{proof}

\begin{proof}[Proof of Theorem~\ref{th:global_max}]
We will work with the BBM with spine.  The basic idea is that the global maximum and the argmax should be roughly equal to the global maximum and the argmax of the spine.  Let $(K_r)_{r \geq 0}$ be a Brownian taboo process started from $K_0 = (L_t-x)/L_t$, and let $${\tilde X}_{\xi}(s) = L_t(s)(1-K_{\tau(s)})$$ denote the position of the spine at time $s$.  Let $5/6 < b < 1$, and define
$${\tilde M} = \max_{s \in [0, t^b]} {\tilde X}_{\xi}(s), \qquad {\tilde m} = \operatorname*{arg\,max}_{s \in [0, t^b]} \: {\tilde X}_{\xi}(s).$$
We will now aim to prove (\ref{maxrem}).  We will first show that (\ref{maxrem}) holds if $\mathfrak M$ and $\mathfrak m$ are replaced by ${\tilde M}$ and ${\tilde m}$.  That is, we will first show that as $t\to\infty$,
\begin{equation}\label{spineMm}
\bigg( L_t^{-1/2} (L_t - {\tilde M}), L_t^{1/2} \frac{{\tilde m}}{t} \bigg) \Rightarrow (R, 3UR).
\end{equation}

Note that
$$L_t - {\tilde X}_{\xi}(s) = L_t - L_t(s) (1-K_{\tau(s)}) = L_t \bigg( \bigg(1 - \frac{L_t(s)}{L_t} \bigg) + \frac{L_t(s)}{L_t} K_{\tau(s)}\bigg).$$
Dividing both sides by $L_t$ and making the substitution $r = \tau(s)$ on the right-hand side, we have
$$\frac{L_t - {\tilde X}_{\xi}(s)}{L_t} = \bigg(1 - \frac{L_t(\tau^{-1}(r))}{L_t} \bigg) + \frac{L_t(\tau^{-1}(r))}{L_t} K_r.$$
Now define
$$\gamma = \frac{\pi^2}{2L_t}.$$
We will apply Lemma~\ref{lem:taboo_with_drift} with $g(\gamma) = \tau(t^b)$ (suppose $t\ge 1$ without loss of generality, so that $t^b \le t$). Indeed, it follows from \eqref{tauasymp} and the hypothesis on $b$ that, as $t\to\infty$ (equivalently, as $\gamma\to 0$),
\[
\gamma^{-1/2} \ll \tau(t^b) \ll \gamma^{-1}.
\]
In particular, hypothesis \eqref{gcond} from Lemma~\ref{lem:taboo_with_drift} is verified.
Furthermore, define for $r\in [0,\tau(t^b)]$, 
\[
b_{\gamma}(r) = \frac{L_t(\tau^{-1}(r))}{L_t}, \qquad d_{\gamma}(r) = 1 - \frac{L_t(\tau^{-1}(r))}{L_t}.
\]
Then
\begin{equation}\label{spinemotion}
\frac{L_t - {\tilde X}_{\xi}(s)}{L_t} = b_{\gamma}(r) K_r + d_{\gamma}(r).
\end{equation}
It follows that
\begin{equation}\label{LtMm}
\bigg(\frac{L_t - {\tilde M}}{L_t}, \tilde m \bigg) = \bigg( \min_{r \in [0, \tau(t^b)]} \big( b_{\gamma}(r) K_r + d_{\gamma}(r) \big), \: \tau^{-1} \Big( \operatorname*{argmin}_{r \in [0, \tau(t^b)]} \big( b_{\gamma}(r) K_r + d_{\gamma}(r) \big) \Big) \bigg).
\end{equation}
Because $\tau(t^b) \ll t^{1/3}$, we can see from (\ref{Ltdiff}) that
$$\lim_{t \rightarrow \infty} \sup_{r \in [0, \tau(t^b)]} |1 - b_{\gamma}(r)| = 0$$
and $$\lim_{t \rightarrow \infty} \sup_{r \in [0, \tau(t^b)]} \bigg|1 - \frac{d_{\gamma}(r)}{\gamma r} \bigg| = 0.$$  Because $\gamma \rightarrow 0$ as $t \rightarrow \infty$, the hypotheses (\ref{adhyp}) in Lemma \ref{lem:taboo_with_drift} are satisfied.  Because $L_t - x \gg t^{1/6}$ by assumption, we have $K_0 \gg t^{-1/6}$, and therefore $K_0 \gg \sqrt{\gamma}$, so the assumption (\ref{initialz}) in Lemma~\ref{lem:taboo_with_drift} also holds.  Thus, by (\ref{Rayleighconv}),
\begin{equation}\label{newRUR}
\bigg( \frac{1}{\sqrt{\gamma}} \min_{r \in [0, \tau(t^b)]} \big( b_{\gamma}(r) K_r + d_{\gamma}(r) \big), \sqrt{\gamma} \operatorname*{argmin}_{r \in [0, \tau(t^b)]} \big( b_{\gamma}(r) K_r + d_{\gamma}(r) \big) \bigg) \Rightarrow ({\tilde R}, U {\tilde R}).
\end{equation}
Combining (\ref{LtMm}) with (\ref{newRUR}) and using (\ref{tauinverse}), we get, as $t\to\infty$,
\begin{equation}\label{LtMprelim}
\bigg( \frac{L_t - {\tilde M}}{L_t \sqrt{\gamma}}, \frac{\sqrt{\gamma} {\tilde m}}{L_t^2} \bigg) \Rightarrow ({\tilde R}, U {\tilde R}).
\end{equation}
Note that $L_t \sqrt{\gamma} = \sqrt{\pi^2 L_t/2}$, and recall from (\ref{gammaLt2}) that
  $$\frac{\sqrt{\gamma}}{L_t^2} = \sqrt{\frac{2}{\pi^2}} \cdot \frac{\sqrt{L_t}}{3t}.$$  Therefore, we can rewrite (\ref{LtMprelim}) as
$$\bigg( L_t^{-1/2} (L_t - {\tilde M}), L_t^{1/2} \frac{{\tilde m}}{t} \bigg) \Rightarrow \bigg( \sqrt{\frac{\pi^2}{2}} {\tilde R}, 3 \sqrt{\frac{\pi^2}{2}} U {\tilde R} \bigg),\quad t\to\infty.$$
Because ${\tilde R} \sqrt{\pi^2/2}$ has the same distribution as $R$, the result (\ref{spineMm}) follows.

The next step is to extend the result to the case of the full BBM with spine, rather than considering only the spinal particle.  Let
$$\tilde{\mathfrak M} = \max_{s \in [0, t^b]} {\tilde X}_1(s), \qquad \tilde{\mathfrak m} = \operatorname*{arg\,max}_{s \in [0, t^b]} {\tilde X}_1(s).$$  
To compare $\tilde{\mathfrak M}$ with ${\tilde M}$, let $0 < \delta < 1/6$, and $A_1$ be the event that there exist $s \in (0, t^b)$ and $z > 0$ such that a particle branches off the spine at time $s$ and location $z$, and eventually has a descendant that gets above $z + t^{\delta}$.  To bound $P(A_1)$, recall that the spine branches at rate $(m+1)\beta$ according to the size-biased offspring distribution.  Let $\zeta$ denote the mean of the size-biased offspring distribution, which is finite because the offspring distribution has finite variance.  Then the expected number of children of the spine by time $t^b$ is $\beta(m+1)\zeta t^b$.  By Lemma \ref{MGbound}, the probability that a particular one of these children has a descendant that gets more than $t^{\delta}$ above the location where it branched off the spine is at most $e^{-t^{\delta}}$.  Thus,
\begin{equation}\label{descendant}
P(A_1) \leq \beta (m+1) \zeta t^b \cdot e^{-t^{\delta}},
\end{equation}
which tends to zero as $t \rightarrow \infty$.  Because, on the event $A_1$, we have $0 \leq \tilde{\mathfrak M} - {\tilde M} \leq t^{\delta}$, it follows that
\begin{equation}\label{changeM}
L_t^{-1/2}(\tilde{\mathfrak M} - {\tilde M}) \rightarrow_p 0,
\end{equation}
where $\rightarrow_p$ denotes convergence in probability as $t \rightarrow \infty$.

We next derive a similar result for the argmax. Recall that $\tilde m$ is the time when the spine attains its maximum and note that equation (\ref{spineMm}) implies that ${\tilde m} = \Theta_p(t^{5/6})$, and therefore (\ref{tauasymp}) implies that $\tau({\tilde m}) = \Theta_p(t^{1/6})$. We aim to show that
\begin{equation}\label{changem}
t^{-5/6}({\tilde{\mathfrak m}} - {\tilde m}) \rightarrow_p 0,
\end{equation}
We do this in two steps. Define $\hat m \le \tilde{\mathfrak m}$ to be the time at which the particle that attains the maximum branches off the spine. We bound seperately $|\hat m - \tilde m|$ and $\tilde{\mathfrak m} - \hat m$, starting with the first one. Let $\kappa > 0$. By \eqref{secondexc} from Lemma~\ref{lem:taboo_with_drift} and \eqref{spinemotion}, there exists $\eta>0$, such that for every $\theta > 0$, $$\limsup_{t \rightarrow \infty} \P_x\bigg( \tilde{X_{\xi}}({\tilde m}) - \sup_{\substack{s \in [0,\tau(t^b)]\\ |\tau(s)-\tau(\tilde m)| \ge \theta/\sqrt\gamma}} \tilde{X}_{\xi}(s) \leq \eta \sqrt{\gamma} L_t  \bigg) < \kappa.$$
Because $\sqrt{\gamma} L_t \asymp t^{1/6}$, it thus follows from (\ref{descendant}) that with probability tending to one as $t \rightarrow \infty$, 
\[
|\tau(\hat m)-\tau(\tilde m)| < \theta/\sqrt{\gamma}.
\]
But since $\tau(\tilde m) = \Theta_p(t^{1/6})$ and $\sqrt\gamma = \Theta_p(t^{-1/6})$, and since $\theta>0$ is arbitrary, it follows from \eqref{tauinverse} that 
\begin{equation}\label{changemhat}
t^{-5/6}(\hat m - \tilde m) \rightarrow_p 0.
\end{equation}

It remains to bound the difference $\tilde{\mathfrak m} - \tilde m$.  Choose $d$ such that $2/3 < d < 5/6$, and let $u = t^d$.  Let $A_2$ 
be the event that there exist $s \in (0, t^b)$ and $z > 0$ such that a particle branches off the spine at time $s$ and location $z$, and a descendant of this particle gets above $z$ after time $t + u$.
In particular, we have
\begin{equation}
\label{A2}
\tilde{\mathfrak{m}} - \hat m \le t^d,\quad\text{on the event $A_2^c$}.
\end{equation}
To bound the probability of the event $A_2$, suppose a particle branches off the spine at time $s$ and location $z$, and choose $v$ such that $z = L_v$.  By Lemma \ref{rtmostms}, the probability that any descendant of the particle gets above $z$ after time $s + u$ is at most $C_0 (L_{v + u} - z) e^{-(L_{v+u} - z)}$.  Note that $L_{v+u} - z = L_{v+u} - L_v$, which is a decreasing function of $v$.  Because the spine never gets above $L_t$, the expression $L_{v+u} - z$ is minimized when $z = L_t$, or equivalently when $v = t$.  Therefore, the probability that any descendant of the particle gets above $z$ after time $s + u$ is at most $C_0 (L_{t+u} - L_t) e^{-(L_{t+u} - L_t)}$.  Thus,
\begin{equation}\label{PA2}
P(A_2) \leq \beta (m+1) \zeta t^b \cdot C_0 (L_{t+u} - L_t) e^{-(L_{t+u} - L_t)}.
\end{equation}
Because $L_{t+u} - L_t \asymp t^{d - 2/3}$, the above expression tends to zero as $t \rightarrow \infty$. Together with \eqref{changemhat} and \eqref{A2}, this proves \eqref{changem}. Then, from (\ref{spineMm}), (\ref{changeM}), and (\ref{changem}), we get
\begin{equation}\label{BBMwspine}
\bigg( L_t^{-1/2} (L_t - \tilde{\mathfrak M}), L_t^{1/2} \frac{\tilde{\mathfrak m}}{t} \bigg) \Rightarrow (R, 3UR).
\end{equation}

By Proposition \ref{lem:spine_comparison}, the result (\ref{BBMwspine}) still holds when we replace the BBM with spine by the original branching Brownian motion conditioned on $\zeta > t$ in the definitions of $\tilde{\mathfrak M}$ and $\tilde{\mathfrak m}$.
It remains to show that for the branching Brownian motion conditioned on $\zeta > t$, with high probability the maximum will not occur after time $t^b$.  However, note that (\ref{spineMm}) implies that $L_t - \tilde{\mathfrak M} = \Theta_p(t^{1/6})$.  Therefore, we can choose $a$ such that $1/6 < a < b - 2/3$, and then apply Lemma \ref{rightLD} with these choices of $a$ and $b$.  Lemma \ref{rightLD} implies that, with probability tending to one as $t \rightarrow \infty$, no particle gets above $L_t - t^a$ after time $t^b$, which means that with probability tending to one as $t \rightarrow \infty$, the maximum occurs before time $t^b$.  The proof is now complete.
\end{proof}

\bigskip
\noindent {\bf {\Large Acknowledgments}}

\bigskip
\noindent We thank Julien Berestycki for helpful discussions related to this work.

\end{document}